\documentclass[12pt]{amsart}

  \baselineskip 1 mm

\usepackage{amsmath,amsthm,amsfonts,latexsym,amsopn,verbatim,amscd,amssymb}
\usepackage{hyperref}
\usepackage{graphics,color}
\usepackage{graphicx}
\usepackage{array,arydshln}
\usepackage{amsbsy}
\usepackage{pst-all}
\usepackage{pstricks}
\usepackage{graphics}
\usepackage{float}
\theoremstyle{plain}
\newtheorem{Thm}{Theorem}[section]
\newtheorem{Lem}[Thm]{Lemma}

\newtheorem{Cor}[Thm]{Corollary}
\theoremstyle{definition}
\newtheorem{Rem}[Thm]{Remark}

\newtheorem{Def}[Thm]{Definition}

\numberwithin{equation}{section}

\newcommand{\bnum}{\begin{enumerate}}
\newcommand{\enum}{\end{enumerate}}

\begin{document}
\title{Nil Clean Divisor Graph}
\author[A. Sharma and D. K. Basnet]{ Ajay Sharma and Dhiren Kumar Basnet*}
\address{\noindent A. Sharma, D. K. Basnet \newline Department of Mathematical Sciences,\newline Tezpur University,\newline  Napaam-784028, Sonitpur,\newline Assam, India.\newline Phone No: +916900740324}
\email{ajay123@tezu.ernet.in and dbasnet@tezu.ernet.in}

\keywords {nil clean ring.\\
\,\,* \emph{Corresponding Author}}

%


%
%
\begin{abstract}
 In this article, we introduce a new graph theoretic structure associated with a finite commutative ring, called nil clean divisor graph. For a ring $R$, nil clean divisor graph is denoted by $G_N(R)$, where the vertex set is $\{x\in R\,:\, x\neq 0, \,\exists\, y(\neq 0, \neq x)\in R$ such that $xy$ is nil clean$\}$, two vertices $x$ and $y$ are adjacent if $xy$ is a nil clean element. We prove some interesting results of nil clean divisor graph of a ring.
\end{abstract}

\maketitle

%
%

\section{Introduction} \label{S:intro}
In this article, rings are finite commutative rings with non zero identity. Diesl \cite{diesl}, introduced the concept of nil clean ring as a subclass of clean ring in $2013$. He defined that an element $x$ of a ring $R$ to be a nil clean element if it can be written as a sum of an idempotent element and a nilpotent element of $R$. $R$ is called nil clean ring if every element of $R$ is nil clean. Also in $2015$, Kosan and Zhou \cite{TK}, developed the concept of weakly nil clean ring as a generalization of nil clean ring. An element $x$ of a ring $R$ is weakly nil clean if $x=n+e$ or $x=n-e$, where $n$ is a nilpotent element and $e$ is an idempotent element of $R$. The set of nilpotent elements, set of unit elements, nil clean elements and weakly nil clean elements of a ring $R$ are denoted by $Nil(R)$, $U(R)$, $NC(R)$ and $WNC(R)$ respectively. By graph, we consider simple undirected graph. For a graph $G$, the set of edges and the set of vertices are denoted by $E(G)$ and $V(G)$ respectively. The concept of zero-divisor graph of a commutative ring was introduced by Beck in \cite{Beck} to discuss the coloring of rings. In $1999$, Anderson and Livingston \cite{al}, introduced zero divisor graph $\Gamma (R)$ of a commutative ring $R$. They defined, the vertex set of $\Gamma (R)$ to be the set of all non-zero zero divisors of $R$ and two vertices $x$ and $y$ are adjacent if $xy=0$. Li et al.\cite{LI}, developed a kind of graph structure of a ring $R$, called nilpotent divisor graph of $R$, whose vertex set is $\{x\in R\,:\,x\neq 0, \, \exists \, y(\neq 0)\in R$ such that $xy\in Nil(R)\}$ and two vertices $x$ and $y$ are adjacent if $xy\in Nil(R)$. In $2018$, Kimball and LaGrange \cite{kimball}, generalized the concept of zero divisor graph to idempotent divisor graph. For any idempotent $e\in R$, they defined the idempotent divisor graph $\Gamma_e(R)$ associated with $e$, where $V(\Gamma_e(R))=\{a\in R\,\,:\,\,$ there exists $b\in R$ with $ab=e\}$ and two vertices $a$ and $b$ are adjacent if $ab=e$.


In this article, we introduce nil clean divisor graph $G_N(R)$ associated with a finite commutative ring $R$. We define the nil clean divisor graph $G_N(R)$ of a ring $R$ by taking $V(G_N(R))=\{x\in R\,:\,x\neq 0,\,\exists\, y(\neq 0, \neq x)\in R$ such that $xy\in NC(R)\}$  as the vertex set and two vertices $x$ and $y$ are adjacent if and only if $xy$ is a nil clean element of $R$. Clearly nil clean divisor graph is a generalization of both idempotent divisor graph and nilpotent divisor graph. The properties like girth, clique number, diameter and dominating number etc. of $G_N(R)$ have been studied.\par

 To start with, we recall some preliminaries about graph theory. For a graph $G$, the degree of a vertex $v\in G$ is the number of edges incident to $v$, denoted by $deg(v)$. The neighbourhood of a vertex $v\in G$ is the set of all vertices incident to $v$, denoted by $A_v$. A graph $G$ is said to be connected, if for any two distinct vertices of $G$, there is a path in $G$ connecting them. Number of edges on the shortest path between vertices $x$ and $y$ is called the distance between $x$ and $y$ and is denoted by $d(x,y)$. If there is no path between $x$ and $y$, then we say $d(x,y)= \infty$. The diameter of a graph $G$, denoted by $diam(G)$, is the maximum of distances of each pair of distinct vertices in $G$. If $G$ is not connected, then we say  $diam(G)=\infty$. Also girth of $G$ is the length of the shortest cycle in $G$, denoted by $gr(G)$ and if there is no cycle in $G$, then we say $gr(G)=\infty$. A complete graph is a simple undirected graph in which every pair of distinct vertices is connected by an edge. 
A clique is a subset a of set of vertices of a graph such that its induced subgraph is complete. A clique having $n$ number of vertices is called an $n$-clique. The maximal clique of a graph is a clique such that there is no clique with more vertices. The clique number of a graph $G$ is denoted by $\omega (G)$ and defined as the number of vertices in a maximal clique of $G$.  
\section{Nil clean divisor graph}
We introduce nil clean divisor graph as follows:
\begin{Def}
  For a ring $R$, nil clean divisor graph, denoted by $G_N(R)$ is defined as a graph with vertex set $\{x\in R\,:\, x\neq 0,\,\exists\, y(\neq 0, \neq x)\in R$ such that $xy \in NC(R)\}$ and two vertices $x$ and $y$ are adjacent if $xy\in NC(R)$.
\end{Def}
From the above definition, we observe that nil clean divisor graph is a generalization of nilpotent divisor graph, which is again a generalization of zero divisor graph. For any idempotent $e\in R$, nil clean divisor graph of $R$ is also a generalization of $\Gamma_e(R)$. As an example, the nil clean divisor graph $G_N(\mathbb{Z}_6)$ is shown below:

\begin{figure}[H]  
\begin{pspicture}(0,3)(0,-1)
\scalebox{.8}{
\rput(5,0){
\psdot[linewidth=.05](-6,3)
\psdot[linewidth=.05](-6,0)
\psdot[linewidth=.05](-9,1.5)
\psdot[linewidth=.05](-3,2.5)
\psdot[linewidth=.05](-3,0.5)
\psline(-9,1.5)(-6,3)(-6,0)(-9,1.5)(-6,0)(-9,1.5)(-3,2.5)(-3,0.5)(-9,1.5)
\rput(-9.3,1.5){$3$}
\rput(-6,3.3){$4$}
\rput(-6,-0.3){$1$}
\rput(-3,2.8){$2$}
\rput(-3,0.2){$5$}
}}
\end{pspicture}
\caption{Nil clean divisor graph of $\mathbb{Z}_6$.}\label{1}
\end{figure}
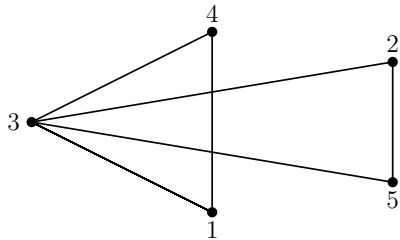

\begin{Thm}
  The nil clean divisor graph $G_N(R)$ is complete if and only if $R$ is a nil clean ring.
\end{Thm}
\begin{proof}
  Let $G_N(R)$ is a complete and $x\in R$. If $x=0$, then $x$ is nil clean, if $x\neq 0$ then $x.1=x$ is nil clean as $1\in V(G_N(R))$. Converse is clear from the definition of nil clean divisor graph.
\end{proof}

  If $\mathbb{F}$ is a finite field of order $n$, then clearly $NC(\mathbb{F})=\{0,1\}$. Hence for any $x(\neq 0)\in \mathbb{F}$, $x$ is adjacent to only $x^{-1}$, provided $x\neq x^{-1}$. Hence the nil-clean divisor graph of $\mathbb{F}$ is as follows:
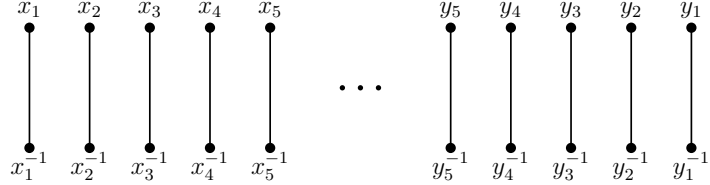
\begin{figure}[H]  
\begin{pspicture}(0,3)(0,-1)
\scalebox{.8}{
\rput(5,0){
\psdot[linewidth=.05](-8,2.5)
\rput(-8,2.8){$x_3$}
\psdot[linewidth=.05](-8,0.5)
\rput(-8, .2){$x_3^{-1}$}
\psline(-8,2.5)(-8,.5)
\psline(-9,2.5)(-9,.5)
\psline(-10,2.5)(-10,.5)
\psline(-7,2.5)(-7,.5)
\psline(-6,2.5)(-6,.5)
\psline(-3,2.5)(-3,.5)
\psline(-2,2.5)(-2,.5)
\psline(-1,2.5)(-1,.5)
\psline(0,2.5)(0,.5)
\psline(1,2.5)(1,.5)

\psdot[linewidth=.05](-9,0.5)
\rput(-9, 0.2){$x_2^{-1}$}
\psdot[linewidth=.05](-10,0.5)
\rput(-10, 0.2){$x_1^{-1}$}
\psdot[linewidth=.05](-7,0.5)
\rput(-7, 0.2){$x_4^{-1}$}
\psdot[linewidth=.05](-6,0.5)
\rput(-6, 0.2){$x_5^{-1}$}
\psdot[linewidth=.05](-9,2.5)
\rput(-9,2.8){$x_2$}
\psdot[linewidth=.05](-10,2.5)
\rput(-10,2.8){$x_1$}
\psdot[linewidth=.05](-7,2.5)
\rput(-7,2.8){$x_4$}
\psdot[linewidth=.05](-6,2.5)
\rput(-6,2.8){$x_5$}

\psdot[linewidth=.009](-4.8,1.5)
\psdot[linewidth=.009](-4.5,1.5)
\psdot[linewidth=.009](-4.2,1.5)
\psdot[linewidth=.05](-3,2.5)
\rput(-3,2.8){$y_5$}
\psdot[linewidth=.05](-2,2.5)
\rput(-2,2.8){$y_4$}
\psdot[linewidth=.05](-1,2.5)
\rput(-1,2.8){$y_3$}
\psdot[linewidth=.05](0,2.5)
\rput(0,2.8){$y_2$}
\psdot[linewidth=.05](1,2.5)
\rput(1,2.8){$y_1$}

\psdot[linewidth=.05](-3,0.5)
\rput(-3,0.2){$y_5^{-1}$}
\psdot[linewidth=.05](-2,0.5)
\rput(-2,0.2){$y_4^{-1}$}
\psdot[linewidth=.05](-1,0.5)
\rput(-1,0.2){$y_3^{-1}$}
\psdot[linewidth=.05](0,0.5)
\rput(0,0.2){$y_2^{-1}$}
\psdot[linewidth=.05](1,0.5)
\rput(1,0.2){$y_1^{-1}$}
}}
\end{pspicture}
\caption{Nil clean divisor graph of $\mathbb{F}$.}\label{2}
\end{figure}
Note that $x_i\neq x_i^{-1}$ and $y_i\neq y_i^{-1}$, otherwise we may get some isolated point as well in the graph.
\begin{Cor}
  For a field $\mathbb{F}$ of order $n$, where $n>2$. If $A=\{a\in \mathbb{F}\,\,:\,\, a=a^{-1}\}$ then the following hold.
  \begin{enumerate}
    \item Diameter of $\mathbb{F}$ is infinite.
    \item $Gr(G_N(\mathbb{F}))=\infty$ and $\omega(G_N(\mathbb{F}))=2$.
    \item $|V(G_N(F))|=n-|A|-1$.
  \end{enumerate}
\end{Cor}
\begin{Thm}
  If $R$ has a non trivial idempotent or non trivial nilpotent element, then the girth of $G_N(R)$ is $3$.
\end{Thm}
\begin{proof}
  If $R$ has a non trivial idempotent $e$, then $\{0,1,e,1-e\}\subset  NC(R)$ and we get a cycle $1 - e - (1-e) - 1$. Also if  $R$ has a non trivial nilpotent $n$, then $\{0,1,n,n+1\} \subset NC(R)$. In this case $1 - n - (n+1) - 1$ is a cycle in $G_N(R)$.
\end{proof}
\begin{Thm}
  If $R$ has only trivial idempotents and trivial nilpotent, then girth of $G_N(R)$ is infinite.
\end{Thm}
\begin{proof}
  Since $R$ has only trivial idempotents and trivial nilpotent so by Lemma $2.6$ \cite{ncg}, $R$ is a field. Hence the result.
\end{proof}
\begin{Thm}\label{T2.6}
Let $R$ be a ring. Then the following hold.
\begin{enumerate}
  \item Either $R$ is a field or $G_N(R)$ is connected.
  \item $diam(R)=\infty$ or $diam(R)\leq 3$.
  \item $gr(G_N(R))=\infty$ or $gr(G_N(R))=3$.
\end{enumerate}
\end{Thm}
\begin{proof}
  Suppose $R$ is a reduced ring. \\
  Case (I): If $R$ has no non trivial idempotent, then $R$ is a field. \\
  Case (II): If $R$ has a non trivial idempotent, say $e\in Idem(R)$, then for any $x,y\in V(G_N(R))$, there exist $x_1,y_1\in V(G_N(R))$, such that $xx_1, yy_1\in NC(R)=Idem(R)$. So, we have a path $x-x_1e-y_1(1-e)-y$ from $x$ to $y$.\par
  If $R$ is not a reduced ring, then there exists $n\in Nil(R)$, such that $x-n-y$ is a path from $x$ to $y$, for any $x,y\in V(G_N(R))$. Hence (1) and (2) follow from the above observations and Figure \ref{2}.\\
  (3) If $R$ is reduced, then either $R$ is a field or there exists a non trivial idempotent $e\in R$, such that $1-e-(1-e)-1$ is a cycle. So, $gr(G_N(R))=\infty$ or $gr(G_N(R))=3$. If $R$ is a non reduced ring, then since nilpotent graph is a subgraph of nil clean divisor graph, so from Theorem $2.1$ \cite{LI}, $gr(G_N(R))=3$.

\end{proof}
\begin{Cor}
  If $R$ is not a reduced ring, then $diam(R)\leq 2$.
\end{Cor}
\begin{Cor}
  A ring $R$ is a field if and only if nil clean divisor graph of $R$ is bipartite.
\end{Cor}
\begin{proof}
  $\Rightarrow$ Trivial.\\
  $\Leftarrow$ If nil clean divisor graph of $R$ is bipartite then $gr(G_N(R))\neq 3$. So from Theorem \ref{T2.6}, $gr(G_N(R))=\infty$ and hence $R$ is a field.
\end{proof}
\begin{Thm}
  For a ring $R$, the following are equivalent.
  \begin{enumerate}
    \item $G_N(R)$ is a star graph.
    \item $R \cong \mathbb{Z}_5$.
  \end{enumerate}
\end{Thm}
\begin{proof}
  The result follows from the fact that $gr(G_N(R))=\infty$ if and only if $R$ is a field.
\end{proof}
\begin{Thm}
 For any ring $R$, $\omega(G_N(R))\geq$ max$\{|Nil(R)|, |Idem(R)|-1\}$.
\end{Thm}
\begin{proof}
  From the definition of nil clean divisor graph, we observe that $Nil(R)$ and $Idem(R)$ respectively induces a complete subgraph of $G_N(R)$.
\end{proof}
Next we strudy about nil clean divisor graph of weakly nil clean ring.
\begin{Thm}
  Let $R$ be a weakly nil clean ring which is not nil clean. Then $\omega(G_N(R))\geq [\frac{|R|}{2}]$ and $diam(R)=2$ if $|R|(>3)$ is even, where $[x]$ is the greatest integer function.
\end{Thm}
\begin{proof}
  As $x\in WNC(R)$ implies $-x\in NC(R)$, so if $|R|$ is even, then $|NC(R)|\geq \frac{|R|}{2}$ and if $|R|$ is odd, then $|NC(R)|\geq \frac{|R|+1}{2}$. Since $R$ is commutative, so product of any two nil clean element is also a nil clean element. Hence $\omega(G_N(R))\geq [\frac{|R|}{2}]$.\par
  Since $|R|>3$, so $R$ is not a field and hence $G_N(R)$ is connected. As $|R\setminus \{0\}|$ is odd, so there exists an element $a\in R$ such that $x\in NC(R)\cap WNC(R)$. Hence for any $x,y\in R$, $x-a-y$ is a path in $G_N(R)$ and $diam(G_N(R))= 2$ as $R$ is not a nil clean ring.
\end{proof}
\section{Nil clean divisor graph of $\mathbb{Z}_{2p}$ and $\mathbb{Z}_{3p}$, for any odd prime $p$}
In this section we study the structures of $G_N(\mathbb{Z}_{2p})$ and $G_N(\mathbb{Z}_{3p})$, for any odd prime $p$.
\begin{Lem}\label{3.1}
 If $a\in V(G_N(\mathbb{Z}_{2p}))$, where $p$ is an odd prime, then the following hold.
  \begin{enumerate}
    \item If $a=p$, then $deg(a)=2p-2$.
    \item If $a\in \{1,p-1,p+1,2p-1\}$, then $deg(a)=2$.
    \item Otherwise $deg(a)=3$
  \end{enumerate}
\end{Lem}
\begin{proof}
Clearly $NC(\mathbb{Z}_{2p})=\{0,1,p,p+1\}$.
  \begin{enumerate}
    \item If $a=p$, then for any $y\in V(G_N(\mathbb{Z}_{2p}))$, either $yp=p$ or $yp=0$. Hence every element of $V(G_N(\mathbb{Z}_{2p}))$ is adjacent to $p$.
    \item  It is easy to observe that, $A_1=\{p, p+1\}$, $A_{p-1}=\{p, 2p-1\}$, $A_{p+1}=\{1, p\}$ and $A_{2p-1}=\{p-1, p\}$.

    \item 
        Let $a\in \mathbb{Z}_{2p}\setminus \{0, 1, p-1, p, p+1, 2p-1\}$. \\
        Case (I): Let $a$ be an even number. If $ax=0$ in $\mathbb{Z}_{2p}$, then it has two solutions $0$ and $p$. If $ax=1$ in $\mathbb{Z}_{2p}$, then it has no solution, since $gcd(2p, a)=2\nmid 1$. If $ax=p$ in $\mathbb{Z}_{2p}$, then also it has no solution, since $gcd(2p, a)=2\nmid p$. If $ax=p+1$ in $\mathbb{Z}_{2p}$, then it has two distinct solutions $x_1$ and $x_2$ in $\mathbb{Z}_{2p}$, since $gcd(2p, a)=2\mid p+1$. Hence we conclude that $A_a=\{p,x_1,x_2\}$.\\
        Case (II): Let $a$ be an odd number. If $ax=0$ in $\mathbb{Z}_{2p}$, then it has a unique solution $x=0$. If $ax=1$ in $\mathbb{Z}_{2p}$, then it has unique odd solution $x=y_1$ in $\mathbb{Z}_{2p}$, since $gcd(2p, a)=1\mid 1$. If $ax=p$ in $\mathbb{Z}_{2p}$, then it has unique solution $x=p$, since $gcd(2p, a)=1\mid p$. If $ax=p+1$ in $\mathbb{Z}_{2p}$, then it has unique even solution $x=y_2$ in $\mathbb{Z}_{2p}$, since $gcd(2p, a)=1\mid p+1$. Hence $A_a=\{p,y_1,y_2\}$\\
        From the above cases it follows $deg(a)=3$.
  \end{enumerate}
\end{proof}
\begin{Rem}\label{3.2}
  In the proof of Lemma \ref{3.1} (3), Case(I), since $ax_1=ax_2$ in $\mathbb{Z}_{2p}$, so $x_1-x_2=0$ or $p$, but $x_1-x_2\neq 0$ as $x_1$ and $x_2$ are distinct. Hence if $x_1$ is odd, then $x_2$ is even and if $x_1$ is even, then $x_2$ is odd.
\end{Rem}
From Lemma \ref{3.1} and Remark \ref{3.2}, for any prime $p>2$, the nil clean divisor graph of $\mathbb{Z}_{2p}$ is the following:
\begin{figure}[H]  
\begin{pspicture}(0,3)(0,-1)
\scalebox{.8}{
\rput(5,0){
\psdot[linewidth=.05](-12,1.5)
\psdot[linewidth=.05](-10,3.5)
\psdot[linewidth=.05](-10,-.5)
\psdot[linewidth=.05](-9,3.5)
\psdot[linewidth=.05](-9,-.5)
\psdot[linewidth=.05](-7,3.5)
\psdot[linewidth=.05](-7,-.5)
\psdot[linewidth=.05](-6,3.5)
\psdot[linewidth=.05](-6,-.5)
\psdot[linewidth=.05](-4,3.5)
\psdot[linewidth=.05](-4,-.5)
\psdot[linewidth=.05](-3,3.5)
\psdot[linewidth=.05](-3,-.5)
\psdot[linewidth=.003](-2.2,2.5)
\psdot[linewidth=.003](-2,2.5)
\psdot[linewidth=.003](-1.8,2.5)
\psdot[linewidth=.003](-2.2,.5)
\psdot[linewidth=.003](-2,.5)
\psdot[linewidth=.003](-1.8,.5)
\psdot[linewidth=.05](-1,3.5)
\psdot[linewidth=.05](-1,-.5)
\psdot[linewidth=.05](0,3.5)
\psdot[linewidth=.05](0,-.5)
\psdot[linewidth=.05](1,3)
\psdot[linewidth=.05](1,0)
\psdot[linewidth=.05](2,2.5)
\psdot[linewidth=.05](2,.5)
\psline (-9,-.5)(-12,1.5)(-10,3.5)(-10,-.5)(-9,-.5)(-9,3.5)(-10,3.5)
\psline (-10,-.5)(-12,1.5)(-9,3.5)
\psline (-6,-.5)(-12,1.5)(-7,3.5)(-7,-.5)(-6,-.5)(-6,3.5)(-7,3.5)
\psline (-7,-.5)(-12,1.5)(-6,3.5)
\psline (-3,-.5)(-12,1.5)(-4,3.5)(-4,-.5)(-3,-.5)(-3,3.5)(-4,3.5)
\psline (-4,-.5)(-12,1.5)(-3,3.5)
\psline (-1,-.5)(-12,1.5)(0,3.5)(0,-.5)(-1,-.5)(-1,3.5)(0,3.5)
\psline (0,-.5)(-12,1.5)(-1,3.5)
\psline (-12,1.5)(1,3)(1,0)(-12,1.5)
\psline (-12,1.5)(2,2.5)(2,.5)(-12,1.5)
\rput(-12.3,1.5){$p$}
\rput(1,-.3){$p+1$}
\rput(1,3.3){$1$}
\rput(2,.2){$p-1$}
\rput(2,2.8){$2p-1$}
\rput(-10,3.8){$c_1$}
\rput(-10,-.8){$a_1$}
\rput(-9,3.8){$d_1$}
\rput(-9,-.8){$b_1$}
\rput(-7,3.8){$c_2$}
\rput(-7,-.8){$a_2$}
\rput(-6,3.8){$d_2$}
\rput(-6,-.8){$b_2$}
\rput(-4,3.8){$c_3$}
\rput(-4,-.8){$a_3$}
\rput(-3,3.8){$d_3$}
\rput(-3,-.8){$b_3$}
\rput(-1,3.8){$c_{\frac{p-3}{2}}$}
\rput(-1,-.8){$a_{\frac{p-3}{2}}$}
\rput(0,3.8){$d_{\frac{p-3}{2}}$}
\rput(0,-.8){$b_{\frac{p-3}{2}}$}
}}
\end{pspicture}
\caption{Nil clean divisor graph of $\mathbb{Z}_{2p}$.}\label{3}
\end{figure}
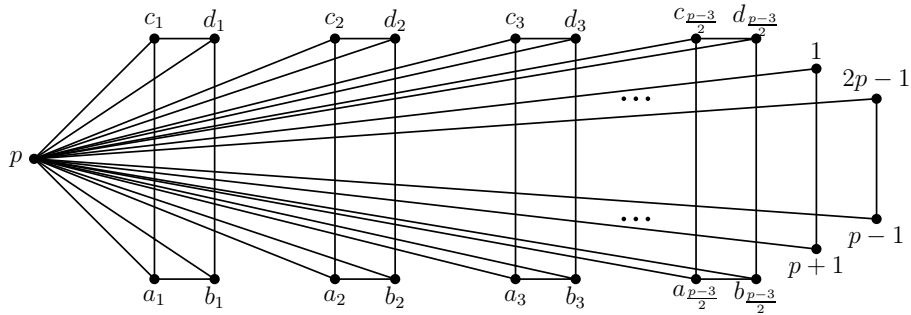
In Figure \ref{3}, $a_i$ and $b_i$ are even numbers from $\mathbb{Z}_{2p}\setminus \{0, 1, p-1, p, p+1, 2p-1\}$ such that $a_ib_i=p+1$, for $1\leq i\leq \frac{p-3}{2}$. Also $c_i=a_i+p$ and $d_i=b_i+p$, for $1\leq i\leq \frac{p-3}{2}$.\\
From the above observations we conclude the following:
\begin{Thm}\label{3.3}
  The following hold for nil clean divisor graph $G_N(\mathbb{Z}_{2p})$, for any odd prime $p$.
  \begin{enumerate}
    \item Clique number of $G_N(\mathbb{Z}_{2p})$ is $3$.
    \item Diameter of $G_N(\mathbb{Z}_{2p})$ is $2$.
    \item Girth of $G_N(\mathbb{Z}_{2p})$ is $3$.
    \item $\{p\}$ is the unique smallest dominating set for $G_N(\mathbb{Z}_{2p})$, that is, dominating number of the graph is $1$.
      \end{enumerate}
\end{Thm}
Next we study about nil clean divisor graph of $\mathbb{Z}_{3p}$. Here we study the graph theoretic properties of $G_N(\mathbb{Z}_{3p})$.
\begin{Lem}\label{3.4}
  In $G_N(\mathbb{Z}_{3p})$; where $p\equiv 2(mod\,3)$, the following hold.
  \begin{enumerate}
    \item $deg(3k)=5$ if $3k\notin \{p+1, 2p-1\}$, for $1\leq k\leq p-1$.
    \item $deg(p+1)=deg(2p-1)=4$.
  \end{enumerate}
\end{Lem}
\begin{proof}
  Here $NC(\mathbb{Z}_{3p})=\{0,1,p+1,2p\}$. Observe that $3k.x\equiv 1(mod \, 3p)$ and $3k.x\equiv 2p(mod \, 3p)$ has no solution, as $gcd(3k,3p)=3$ does not divide $1$ and $2p$. The congruence $3k.x\equiv 0(mod \, 3p)$ has three incongruent solutions $\{0,p,2p\}$ in $\mathbb{Z}_{3p}$. Also $3k.x\equiv p+1(mod\,3p)$ has three distinct incongruent solutions in $\mathbb{Z}_{3p}$, as $gcd(3k,3p)=3$ divides $p+1$.
  \begin{enumerate}
    \item As $x^2\equiv p+1(mod \, 3p)$, has two solutions $p+1$ and $2p-1$, hence if $3k\notin \{p+1, 2p-1\}$, then $deg(3k)=6-1=5$, as $0\notin V(G_N(\mathbb{Z}_{3p}))$.
    \item If $3k\in \{p+1, 2p-1\}$, then $deg(3k)=6-2$, as $0\notin V(G_N(\mathbb{Z}_{3p}))$ and we do not consider any loop.
  \end{enumerate}
\end{proof}
\begin{Lem}\label{3.5}
  In $G_N(\mathbb{Z}_{3p})$, where $p\equiv 2(mod\, 3)$ the following hold.
  \begin{enumerate}
    \item $deg(p)=deg(2p)=2p-2$.
    \item For $x\in \{1,p-1,3p-1,2p+1\}$, $deg(x)=2$.
    \item For $x\in \mathbb{Z}_{3p}\setminus L$, $deg(x)=3$, where $L=\{3k\,\,:\,\,1\leq k\leq p-1\}\cup \{1,p-1,2p+1, 3p-1,p,2p\}$.
  \end{enumerate}
  \end{Lem}
  \begin{proof}Here $NC(\mathbb{Z}_{3p})=\{0,1,p+1,2p\}$.
  \begin{enumerate}
    \item Clearly $p.x\equiv 1(mod\,3p)$ and $p.x\equiv p+1(mod\,3p)$ have no solution as $gcd(3p,p)$ does not divide $1$ and $p+1$. Also $p.x\equiv 0(mod\,3p)$ has $p$ incongruent solutions $\{3k\,\,:\,\, 0\leq k\leq p-1\}$ and $p.x\equiv 2p(mod\,3p)$ has $p$ incongruent solutions $\{3k+2\,\,:\,\, 0\leq k\leq p-1\}$. Since $0\notin V(G_N(\mathbb{Z}_{3p}))$ and $p$ is of the form $3i+2$, for some $0\leq i\leq p-1$, hence $deg(p)=2p-2$. Now $2p.x\equiv 0(mod\,3p)$ has $p$ incongruent solutions $\{3k\,\,:\,\,0\leq k\leq p-1\}$ and $2p.x\equiv 2p(mod\,3p)$ has $p$ incongruent solutions $\{3k+1\,\,:\,\, 0\leq k\leq p-1\}$. But $2p.x\equiv 1(mod\,3p)$ and $2p.x\equiv p+1(mod\,3p)$ have no solutions. Hence $deg(2p)=2p-2$, since $2p$ is of the form $3i+1$, for some $1\leq i \leq p-1$.
    \item Since $x\equiv a(mod\,3p)$, has only one solution $a$, hence $deg(1)=2$. Also $(3p-1).x\equiv c(mod\,3p)$ has only one solution $(3p-1)a$, hence $deg(3p-1)=2$, as $0\notin V(G_N(\mathbb{Z}_{3p}))$ and $3p-1\in U(\mathbb{Z}_{3p})$. Equation $(p-1).x\equiv 1(mod\,3p)$ and $(2p+1).x\equiv c(mod\,3p)$ have a unique solutions, where $c\in \{0,1,2p,p+1\}$. Since $p-1,2p+1\in U(\mathbb{Z}_{3p})$, so $deg(p-1)=deg(2p+1)=2$.
    \item Let $a\in \mathbb{Z}_{3p}\setminus L$. As $gcd(a,3p)=1$, so $a.x\equiv 0(mod\,3p)$ has a unique solution $x=0$. Also $a.x\equiv c(mod\,3p)$, where $c\in \{1,2p,p+1\}$ has a unique solution. Hence $deg(a)=3$.
  \end{enumerate}
  \end{proof}
  From Lemma \ref{3.4} and Lemma \ref{3.5}, for any prime $p>3$ with $p\equiv 2(mod\, 3)$, the nil clean divisor graph of $\mathbb{Z}_{3p}$ is the following:
  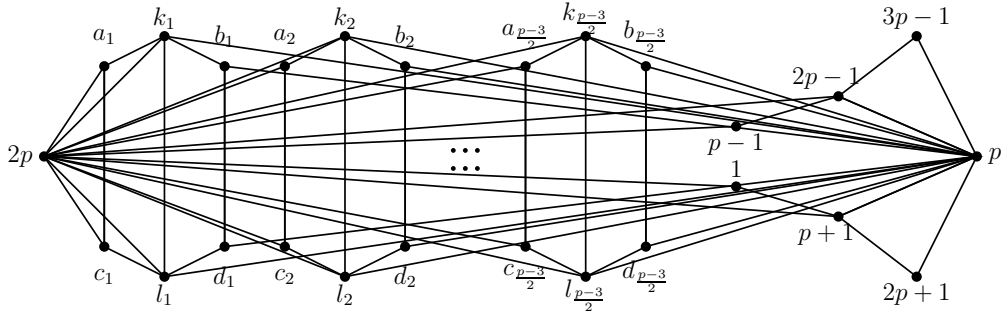
\begin{figure}[H]  
\begin{pspicture}(0,3)(0,-1)
\scalebox{.8}{
\rput(5,0){
\psline (-11, 3)(-11, 0)
\psline (-10, 3.5)(-10, -.5)
\psline (-9, 3)(-9, 0)
\psline (-8, 3)(-8, 0)
\psline (-7, 3.5)(-7, -.5)
\psline (-6, 3)(-6, 0)
\psline (-4, 3)(-4, 0)
\psline (-3, 3.5)(-3, -.5)
\psline (-2, 3)(-2, 0)
\psline (3.5, 1.5)(-10, 3.5)
\psline (3.5, 1.5)(-10, -.5)
\psline (3.5, 1.5)(-7, 3.5)
\psline (3.5, 1.5)(-7, -.5)
\psline (3.5, 1.5)(-3, 3.5)
\psline (3.5, 1.5)(-3, -.5)
\psline (-9, 3)(3.5, 1.5)(-9, 0)
\psline (-6, 3)(3.5, 1.5)(-6, 0)
\psline (-2, 3)(3.5, 1.5)(-2, 0)
\psdot[linewidth=.05](-12,1.5)
\psdot[linewidth=.05](-11, 0)
\psdot[linewidth=.05](-11, 3)
\psdot[linewidth=.05](-10, 3.5)
\psdot[linewidth=.05](-10, -.5)
\psdot[linewidth=.05](-9, 3)
\psdot[linewidth=.05](-9, 0)
\psdot[linewidth=.05](-8, 3)
\psdot[linewidth=.05](-8, 0)
\psdot[linewidth=.05](-7, 3.5)
\psdot[linewidth=.05](-7, -.5)
\psdot[linewidth=.05](-6, 3)
\psdot[linewidth=.05](-6, 0)
\psdot[linewidth=.05](-4, 3)
\psdot[linewidth=.05](-4, 0)
\psdot[linewidth=.05](-3, 3.5)
\psdot[linewidth=.05](-3, -.5)
\psdot[linewidth=.05](-2, 3)
\psdot[linewidth=.05](-2, 0)
\psdot[linewidth=.05](-.5, 2)
\psdot[linewidth=.05](-.5, 1)
\psdot[linewidth=.05](1.2, 2.5)
\psdot[linewidth=.05](1.2, .5)
\psdot[linewidth=.05](2.5, 3.5)
\psdot[linewidth=.05](2.5, -.5)
\psdot[linewidth=.05](3.5, 1.5)
\psdot[linewidth=.005](-5, 1.3)
\psdot[linewidth=.005](-5, 1.6)
\psdot[linewidth=.005](-5.2, 1.3)
\psdot[linewidth=.005](-5.2, 1.6)
\psdot[linewidth=.005](-4.8, 1.3)
\psdot[linewidth=.005](-4.8, 1.6)
\psline (-11,3)(-12,1.5)(-11,0)
\psline (-8,3)(-12,1.5)(-8,0)
\psline (-4,3)(-12,1.5)(-4,0)
\psline (-.5,2)(-12,1.5)(-.5,1)
\psline (1.2,2.5)(-12,1.5)(1.2,.5)
\psline (-10,3.5)(-12,1.5)(-10,-.5)
\psline (-7,3.5)(-12,1.5)(-7,-.5)
\psline (-3,3.5)(-12,1.5)(-3,-.5)
\psline (-11,3)(-10,3.5)(-9,3)(-9,0)(-10,-.5)(-11,0)(-11,3)
\psline (-8,3)(-7,3.5)(-6,3)(-6,0)(-7,-.5)(-8,0)(-8,3)
\psline (-4,3)(-3,3.5)(-2,3)(-2,0)(-3,-.5)(-4,0)(-4,3)
\psline (-.5, 1)(1.2, .5)(3.5, 1.5)(1.2, 2.5)(-.5, 2)
\psline (3.5, 1.5)(1.2, 2.5)(2.5, 3.5)(3.5, 1.5)(1.2, .5)(2.5, -.5)(3.5, 1.5)
\rput(3.8, 1.5){$p$}
\rput(-12.4, 1.5){$2p$}
\rput(-10, 3.8){$k_1$}
\rput(-10, -.8){$l_1$}
\rput(-7, 3.8){$k_2$}
\rput(-7, -.8){$l_2$}
\rput(-3, 3.8){$k_{\frac{p-3}{2}}$}
\rput(-3, -.8){$l_{\frac{p-3}{2}}$}
\rput (-.5, 1.7){$p-1$}
\rput (-.5, 1.3){$1$}
\rput (1, 2.8){$2p-1$}
\rput (1, .2){$p+1$}
\rput (2.5, 3.8){$3p-1$}
\rput (2.5,-.8){$2p+1$}
\rput (-11,3.5){$a_1$}
\rput (-11,-.5){$c_1$}
\rput (-9, 3.5){$b_1$}
\rput (-9, -.5){$d_1$}
\rput (-8, 3.5){$a_2$}
\rput (-8, -.5){$c_2$}
\rput (-6, 3.5){$b_2$}
\rput (-6, -.5){$d_2$}
\rput (-4, 3.5){$a_{\frac{p-3}{2}}$}
\rput (-4, -.5){$c_{\frac{p-3}{2}}$}
\rput (-2, 3.5){$b_{\frac{p-3}{2}}$}
\rput (-2, -.5){$d_{\frac{p-3}{2}}$}
}}
\end{pspicture}
\caption{Nil clean divisor graph of $\mathbb{Z}_{3p}$, where $p\equiv 2(mod \,3)$.}\label{4}
\end{figure}
In Figure \ref{4}, $\{l_i,k_i\}\subseteq \{3k\,\,:\,\,1\leq k\leq p-1\}$, $a_ic_i\equiv 1(mod\,3p)$, $b_id_i\equiv 1(mod\,3p)$ and $a_ik_i \equiv c_il_i\equiv b_ik_i\equiv d_il_i\equiv p+1(mod\, 3p)$, for $1\leq i\leq \frac{p-3}{2}$. Also $a_i\equiv c_i\equiv 1(mod\,3)$ and $b_i\equiv d_i\equiv 2(mod\,3)$, for $1\leq i\leq \frac{p-3}{2}$.
\begin{Thm}\label{3.6}
   For any prime $p$, where $p\equiv 2(mod \,3)$, the following hold for $G_N(\mathbb{Z}_{3p})$.
   \begin{enumerate}
     \item Girth of $G_N(\mathbb{Z}_{3p})$ is $3$.
     \item Clique number of $G_N(\mathbb{Z}_{3p})$ is $3$.
     \item Diameter of $G_N(\mathbb{Z}_{3p})$ is $3$.
     \item $\{p,2p\}$ is the unique smallest dominating set for $G_N(\mathbb{Z}_{3p})$, that is, dominating number of the graph is $2$.
   \end{enumerate}
\end{Thm}
\begin{proof}
  Clearly $NC(\mathbb{Z}_{3p})=\{0,1,p+1,2p\}$.
  \begin{enumerate}
    \item Since $p-(p+1)-(2p+1)-p$ is a cycle of $G_N(\mathbb{Z}_{3p})$, so girth of $G_N(\mathbb{Z}_{3p})$ is $3$.
    \item If possible, let $\omega((G_N(\mathbb{Z}_{3p}))=4$. Then there exists $A=\{a_i\,\,:\,\,1\leq i\leq 4\}\subset V(G_N(\mathbb{Z}_{3p}))$ such that $A$ forms a complete subgraph of $G_N(\mathbb{Z}_{3p})$. If $x\in \mathbb{Z}_{3p}\setminus \{p,2p,3k\,\,:\,\, 1\leq k\leq p-1\}$, then $deg(x)\leq 3$. Also $x$ is adjacent to either $p$ or $2p$, $x^{-1}$ and $3i$, for some $1\leq i\leq p-1$(provided $x\notin \{1,p-1,2p+1, 3p-1\}$). But $x^{-1}$ is also adjacent to $3j$, for some $1\leq j\leq p-1$ such that $i\neq j$. So $A\subseteq \{p,2p,3k\,\,:\,\, 1\leq k\leq p-1\}$. Suppose $a_1=3k$, for some $1\leq k\leq p-1$. From Figure \ref{4}, $A_{a_1}=\{p,2p,3i+1, 3j+2,3s\}$, where $1\leq i,j,s\leq p-1$, also $3s\notin A_{3i+1}$, $3s\notin A_{3j+2}$, $3i+1\notin A_{3j+2}$, $p\notin A_{2p}$, $2p\notin A_{3j+2}$ and $p\notin N_{3i+1}$. Therefore $a_i\notin \{3k\,\,:\,\,1\leq k\leq p-1\}$, a contradiction. Hence $\omega((G_N(\mathbb{Z}_{3p}))=3$, as $\{p, 2p-1,3p-1\}$ forms a complete subgraph of $G_N(\mathbb{Z}_{3p})$.
    \item 
        From Figure \ref{4}; $1$ and $2$ are connected by minimum $3$ edges, so by Theorem \ref{T2.6}, $diam(G_N(\mathbb{Z}_{3p}))=3$.
    \item Since every element of $G_N(\mathbb{Z}_{3p})\setminus \{p,2p\}$ is adjacent to either $p$ or $2p$. Hence proof follows from Figure \ref{4}.
  \end{enumerate}
\end{proof}
\begin{Lem}\label{3.7}
  In $G_N(\mathbb{Z}_{3p})$; where $p\equiv 1(mod\,3)$, the following hold.
  \begin{enumerate}
    \item $deg(3k)=5$ if $3k\notin \{p-1, 2p+1\}$, for $1\leq k\leq p-1$.
    \item $deg(p-1)=deg(2p+1)=4$.
  \end{enumerate}
\end{Lem}
\begin{proof}
  Proof is similar to the proof of Lemma \ref{3.4}.
\end{proof}
\begin{Lem}\label{3.8}
  In $\mathbb{Z}_{3p}$, where $p\equiv 1(mod\,3)$, the following hold.
  \begin{enumerate}
    \item $deg(p)=deg(2p)=2p-2$.
    \item For $x\in \{1,p+1,3p-1,2p-1\}$, $deg(x)=2$.
    \item For $x\in \mathbb{Z}_{3p}\setminus L$, $deg(x)=3$, where $L=\{3k\,\,:\,\,1\leq k\leq p-1\}\cup \{1,p,2p,p+1,2p-1, 3p+1\}$.
  \end{enumerate}
  \end{Lem}
  \begin{proof}
    Proof is similar to the proof Lemma \ref{3.5}.
  \end{proof}
  From Lemma \ref{3.7} and Lemma \ref{3.8}, the nil clean divisor graph of $\mathbb{Z}_{3p}$, where $p\equiv 1(mod\,3)$ is the following:
  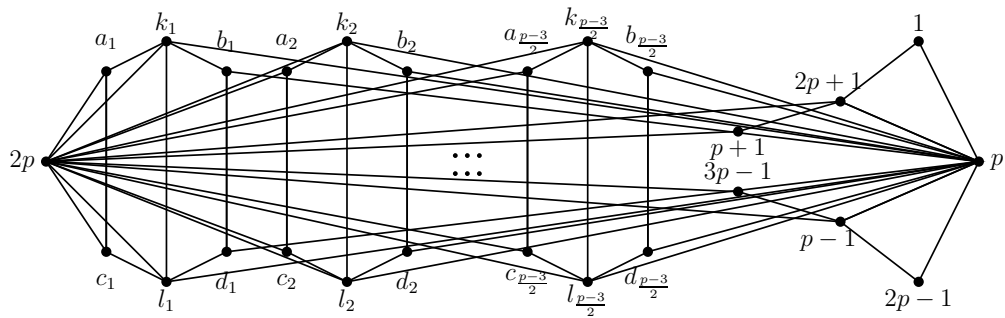
\begin{figure}[H]  
\begin{pspicture}(0,3)(0,-1)
\scalebox{.8}{
\rput(5,0){
\psline (-11, 3)(-11, 0)
\psline (-10, 3.5)(-10, -.5)
\psline (-9, 3)(-9, 0)
\psline (-8, 3)(-8, 0)
\psline (-7, 3.5)(-7, -.5)
\psline (-6, 3)(-6, 0)
\psline (-4, 3)(-4, 0)
\psline (-3, 3.5)(-3, -.5)
\psline (-2, 3)(-2, 0)
\psline (3.5, 1.5)(-10, 3.5)
\psline (3.5, 1.5)(-10, -.5)
\psline (3.5, 1.5)(-7, 3.5)
\psline (3.5, 1.5)(-7, -.5)
\psline (3.5, 1.5)(-3, 3.5)
\psline (3.5, 1.5)(-3, -.5)
\psline (-9, 3)(3.5, 1.5)(-9, 0)
\psline (-6, 3)(3.5, 1.5)(-6, 0)
\psline (-2, 3)(3.5, 1.5)(-2, 0)
\psdot[linewidth=.05](-12,1.5)
\psdot[linewidth=.05](-11, 0)
\psdot[linewidth=.05](-11, 3)
\psdot[linewidth=.05](-10, 3.5)
\psdot[linewidth=.05](-10, -.5)
\psdot[linewidth=.05](-9, 3)
\psdot[linewidth=.05](-9, 0)
\psdot[linewidth=.05](-8, 3)
\psdot[linewidth=.05](-8, 0)
\psdot[linewidth=.05](-7, 3.5)
\psdot[linewidth=.05](-7, -.5)
\psdot[linewidth=.05](-6, 3)
\psdot[linewidth=.05](-6, 0)
\psdot[linewidth=.05](-4, 3)
\psdot[linewidth=.05](-4, 0)
\psdot[linewidth=.05](-3, 3.5)
\psdot[linewidth=.05](-3, -.5)
\psdot[linewidth=.05](-2, 3)
\psdot[linewidth=.05](-2, 0)
\psdot[linewidth=.05](-.5, 2)
\psdot[linewidth=.05](-.5, 1)
\psdot[linewidth=.05](1.2, 2.5)
\psdot[linewidth=.05](1.2, .5)
\psdot[linewidth=.05](2.5, 3.5)
\psdot[linewidth=.05](2.5, -.5)
\psdot[linewidth=.05](3.5, 1.5)
\psdot[linewidth=.005](-5, 1.3)
\psdot[linewidth=.005](-5, 1.6)
\psdot[linewidth=.005](-5.2, 1.3)
\psdot[linewidth=.005](-5.2, 1.6)
\psdot[linewidth=.005](-4.8, 1.3)
\psdot[linewidth=.005](-4.8, 1.6)
\psline (-11,3)(-12,1.5)(-11,0)
\psline (-8,3)(-12,1.5)(-8,0)
\psline (-4,3)(-12,1.5)(-4,0)
\psline (-.5,2)(-12,1.5)(-.5,1)
\psline (1.2,2.5)(-12,1.5)(1.2,.5)
\psline (-10,3.5)(-12,1.5)(-10,-.5)
\psline (-7,3.5)(-12,1.5)(-7,-.5)
\psline (-3,3.5)(-12,1.5)(-3,-.5)
\psline (-11,3)(-10,3.5)(-9,3)(-9,0)(-10,-.5)(-11,0)(-11,3)
\psline (-8,3)(-7,3.5)(-6,3)(-6,0)(-7,-.5)(-8,0)(-8,3)
\psline (-4,3)(-3,3.5)(-2,3)(-2,0)(-3,-.5)(-4,0)(-4,3)
\psline (-.5, 1)(1.2, .5)(3.5, 1.5)(1.2, 2.5)(-.5, 2)
\psline (3.5, 1.5)(1.2, 2.5)(2.5, 3.5)(3.5, 1.5)(1.2, .5)(2.5, -.5)(3.5, 1.5)
\rput(3.8, 1.5){$p$}
\rput(-12.4, 1.5){$2p$}
\rput(-10, 3.8){$k_1$}
\rput(-10, -.8){$l_1$}
\rput(-7, 3.8){$k_2$}
\rput(-7, -.8){$l_2$}
\rput(-3, 3.8){$k_{\frac{p-3}{2}}$}
\rput(-3, -.8){$l_{\frac{p-3}{2}}$}
\rput (-.5, 1.7){$p+1$}
\rput (-.5, 1.3){$3p-1$}
\rput (1, 2.8){$2p+1$}
\rput (1, .2){$p-1$}
\rput (2.5, 3.8){$1$}
\rput (2.5,-.8){$2p-1$}
\rput (-11,3.5){$a_1$}
\rput (-11,-.5){$c_1$}
\rput (-9, 3.5){$b_1$}
\rput (-9, -.5){$d_1$}
\rput (-8, 3.5){$a_2$}
\rput (-8, -.5){$c_2$}
\rput (-6, 3.5){$b_2$}
\rput (-6, -.5){$d_2$}
\rput (-4, 3.5){$a_{\frac{p-3}{2}}$}
\rput (-4, -.5){$c_{\frac{p-3}{2}}$}
\rput (-2, 3.5){$b_{\frac{p-3}{2}}$}
\rput (-2, -.5){$d_{\frac{p-3}{2}}$}
}}
\end{pspicture}
\caption{Nil clean divisor graph of $\mathbb{Z}_{3p}$, where $p\equiv 1(mod \,\,3)$.}\label{5}
\end{figure}
In Figure \ref{5}, $\{l_i,k_i\}\subseteq \{3k\,\,:\,\,1\leq k\leq p-1\}$, $a_ic_i\equiv 1(mod\,\,3p)$, $b_id_i\equiv 1(mod\,\,3p)$ and $a_ik_i \equiv c_il_i\equiv b_ik_i\equiv d_il_i\equiv 2p+1(mod\,\, 3p)$, for $1\leq i\leq \frac{p-3}{2}$. Also $a_i\equiv c_i\equiv 2(mod\,3)$ and $b_i\equiv d_i\equiv 1(mod\,3)$, for $1\leq i\leq \frac{p-3}{2}$. Hence we get the following theorem:
\begin{Thm}\label{3.9}
   The following hold for $G_N(\mathbb{Z}_{3p})$, for any prime $p$, where $p\equiv 1(mod \,\,3)$.
   \begin{enumerate}
     \item Girth of $G_N(\mathbb{Z}_{3p})$ is $3$.
     \item Clique number of $G_N(\mathbb{Z}_{3p})$ is $3$.
     \item Diameter of $G_N(\mathbb{Z}_{3p})$ is $3$.
     \item $\{p,2p\}$ is the unique smallest dominating set for $G_N(\mathbb{Z}_{3p})$, that is, dominating number of the graph is $2$.
   \end{enumerate}
\end{Thm}
\begin{proof}
  Since Figure \ref{4} and Figure \ref{5} are similar, hence the proof is similar to the proof of Theorem \ref{3.6}.
\end{proof}
\section{Acknowledgement}
The first Author was supported by Government of India under DST(Department of Science and Technology), DST-INSPIRE registration no IF160671.


\begin{thebibliography}{30}
\bibitem{al} Anderson, D.F. and Livingston, P.S., The zero-divisor graph of a commutative ring. \textit{J. Algebra}, $217(2):434-447$, $1999$.
\bibitem{ncg}Basnet D. K. and Bhattacharyya J., Nil clean graphs of rings, \textit{Algebra Colloq.}, $24(3):481-492$, $2017$.
\bibitem{Beck} Beck, I., Coloring of commutative rings. \textit{J. Algebra} $116(1):208-226$,  $1988$.
\bibitem{diesl} Diesl, A. J., Nil clean rings. \textit{J. Algebra} $383:197-211$,  $2013$.
\bibitem{gt}Diestel R., \textit{Graph Theory}, Springer-Verlag, New York 1997, electronic edition 2000.
\bibitem{gfr}Grimaldi R. P., Graphs from rings, \textit{Proceedings of the 20th Southeastern
Conference on Combinatorics, Graph Theory, and Computing}, Volume 71, pages $95$–$103$, Florida, USA, $1990$. Atlantic University.
\bibitem{kimball}Kimball, C. F. and LaGrange J. D., The idempotent-divisor graphs of a commutative ring. \textit{Comm. Algebra}, $46(9):3899-3912$, $2018$.
\bibitem{TK}Kosan, M. Tamer, and Zhou, Y., On weakly nil-clean rings. \textit{Frontiers of Mathematics in China} $11.4:949-955$, $2016$.
\bibitem{LI} Li, A. and Li, Q., A kind of graph structure on von-Neumann regular rings. \textit{International J. Algebra} $4:291-302$, $2010$.










\end{thebibliography}
\end{document}